%% file: main.tex
\documentclass[A4paper, 12pt]{article}
\usepackage[utf8]{inputenc}
\usepackage[OT2, T1]{fontenc}
\usepackage{cite}
\usepackage{amsmath}
\usepackage{amsfonts}
\usepackage{amssymb}
\usepackage{amsxtra}
\usepackage{mathrsfs}
\usepackage{tensor}
\usepackage[thmmarks, amsmath, thref, amsthm]{ntheorem}
\usepackage[pagebackref]{hyperref}

\usepackage{graphicx}
\usepackage[all]{xy}
\usepackage[portrait, top=2cm, bottom=2cm, left=2cm, right=2cm]{geometry}
\usepackage[new]{old-arrows} 

\usepackage{enumerate}

\numberwithin{equation}{section}

\theoremstyle{plain}
\newtheorem{thm}{Theorem}
\newtheorem{lemm}[thm]{Lemma}
\newtheorem{prop}[thm]{Proposition}
\newtheorem{coro}[thm]{Corollary}

\theoremstyle{definition}

\newtheorem{remk}[thm]{Remark}

\newcommand{\Br}{\operatorname{Br}}
\newcommand{\Gal}{\operatorname{Gal}}
\newcommand{\Hom}{\operatorname{Hom}}
\newcommand{\Img}{\operatorname{Im}}
\newcommand{\Ind}{\operatorname{Ind}}
\newcommand{\Inv}{\operatorname{Inv}}
\newcommand{\inv}{\operatorname{inv}}
\newcommand{\Ker}{\operatorname{Ker}}
\newcommand{\Mat}{\operatorname{Mat}}
\newcommand{\nr}{\operatorname{nr}}
\newcommand{\SL}{\mathbf{SL}}

\newcommand{\Fbf}{\mathbf{F}}
\newcommand{\fbf}{\mathbf{f}}
\newcommand{\Gbf}{\mathbf{G}}
\newcommand{\Ibf}{\mathbf{I}}
\newcommand{\nbf}{\mathbf{n}}
\newcommand{\Qbf}{\mathbf{Q}}
\newcommand{\Zbf}{\mathbf{Z}}

\newcommand{\Acal}{\mathcal{A}}
\newcommand{\Bcal}{\mathcal{B}}

\renewcommand{\H}{\mathrm{H}}
\newcommand{\B}{\mathrm{B}}

\DeclareSymbolFont{cyrletters}{OT2}{wncyr}{m}{n}
\DeclareMathSymbol{\Sha}{\mathalpha}{cyrletters}{"58}

\title{Homogeneous spaces with transcendental Brauer--Manin obstruction to the Hasse principle}
\author{Nguy\~{\^{e}}n M\d{a}nh Linh}
\date{\today}

\begin{document}
\maketitle

\begin{abstract}
    Over any number field with sufficiently many roots of unity, we construct homogeneous spaces of $\mathbf{SL}_n$ with finite geometric stabilizers of nilpotency class $2$, whose unramified algebraic Brauer group consists only of constant classes, and which has no rational point but has local points in every completion of the ground field. This provides the first known examples of failures of the Hasse principle for homogeneous spaces of connected linear algebraic groups that are detected only by transcendental Brauer classes. The construction builds on the failure of a global--local principle for cup-product decomposition, based on a previous idea of Borovoi and Kunyavskii.
\end{abstract}

spaces, nonabelian cohomology, rational points.

{\em 2020 Mathematics subject classifications (MSC).} 11R34, 11S90, 14G12.

\tableofcontents

\input{section1}
\input{section2}
\input{section3}

{\em Acknowledgments.} I thank Mikhail Borovoi, Cyril Demarche, David Harari, Giancarlo Lucchini Arteche, and Olivier Wittenberg for helpful discussions, comments, and suggestions. This work was initiated during a visit to the Lodha Mathematical Sciences Institute for the Thematic Program on Rational Points, Algebraic Cycles, and the Local-Global Principle. I also acknowledge support from the Fondation Sciences Math\'ematiques de Paris during my postdoctoral appointment at the Institut de Math\'ematiques de Jussieu–Paris Rive Gauche, CNRS. I am grateful to both institutions for their excellent working conditions.

\bibliographystyle{alpha}
\bibliography{ref}
\end{document}

%% file: section1.tex
\section{Introduction} \label{sec:Intro}


\subsection{Context} \label{subsec:Context}

A (smooth, geometrically integral) variety $X$ over a number field $k$ (let us say, $k = \Qbf$) is said to be a {\em counterexample to the Hasse principle} if $X(k_v) \neq \varnothing$ for every place $v$ of $k$ but $X(k) = \varnothing$. The Hasse principle (or local--global principle) is known to hold for projective quadric hypersurfaces (Hasse--Minkowski, see \cite[Theorem~27.3]{Shimura2010Arithmetic}) and torsors under simply connected semisimple linear algebraic groups (Kneser--Harder--Chernousov, see \cite[Theorems~6.4 and~6.6]{PR1994Group}). Nevertheless, counterexamples over $k = \Qbf$, due to Selmer and to Cassels and Guy, already occur among projective hypersurfaces such as $3x^3 + 4y^3 + 5z^3 = 0$ \cite{Selmer1951Diophantine} and $5x^3 + 9y^3 + 10z^3 + 12t^3 = 0$ \cite{CG1966Hasse}.

The systematic study of defects of the Hasse principle developed extensively during the twentieth century. Remarkably, Manin's foundational work \cite{Manin1971Brauer} introduced the use of cohomological methods such as Brauer groups and global class field theory. Concretely, the (cohomological) {\em Brauer group} $\Br(X):=\H^2_{\textup{\'et}}(X,\Gbf_m)$ may explain why the Hasse principle fails, as follows. Assume $\prod_v X(k_v) \neq \varnothing$ and consider the pairing
\[
    \Br(X) \times \prod_v X(k_v) \longrightarrow \Qbf/\Zbf, \quad (\alpha, (P_v)) \longmapsto \sum_v \inv_v(\alpha(P_v)),
\]
where $\inv_v: \Br(k_v) \hookrightarrow \Qbf/\Zbf$ are the invariant maps from local class field theory. By the local reciprocity law (Albert--Hasse--Brauer--Noether), the pairing vanishes on the constant subgroup 
\[
    \Br_0(X):=\Img(\Br(k) \longrightarrow \Br(X))
\]
and on the diagonal image of $X(k)$ in $\prod_v X(k_v)$. Therefore, if there is no family $(P_v)$ orthogonal to $\Br(X)$ then $X(k) = \varnothing$. In this case, the failure of the Hasse principle is said to be {\em explained by the Brauer--Manin obstruction}. If $X$ is not projective, we replace $\Br(X)$ by its {\em unramified subgroup} $\Br_{\nr}(X) \subseteq \Br(X)$, which is a stable birational invariant, cf. \cite[\S{6.2}]{CTS2021Brauer}.

In general, even this obstruction may fail to control the Hasse principle, e.g., for Skorobogatov's bielliptic surface \cite{Skorobogatov1999Beyond}. A conjecture by Colliot-Th\'el\`ene \cite[Conjecture~14.1.2]{CTS2021Brauer} predicts the sufficiency of the Brauer--Manin obstruction for varieties with ``tame'' geometry ({\em rationally connected} varieties). These include geometrically unirational varieties such as homogeneous spaces of connected linear algebraic groups. This conjecture was established for principal homogeneous spaces by Sansuc \cite[Corollaire~8.7]{Sansuc1981Arithmetique}. More generally, Borovoi established the corresponding Brauer--Manin statement for homogeneous spaces with connected stabilizers, or commutative stabilizers if the ambient group is simply connected semisimple \cite[Theorems~2.2 et~2.3]{Borovoi1996BrauerManin}. For finite noncommutative stabilizers, Colliot-Th\'el\`ene's conjecture remains largely open. It is worth noting that a reduction argument by Demarche and Lucchini Arteche shows that this latter is also the hardest case \cite[Th\'eor\`eme~1.1]{DLA2019Reduction}.

As a matter of fact, Borovoi's theorem asserts that the {\em algebraic subgroup}
\[
    \Br_{\nr,1}(X):=\Ker(\Br_{\nr}(X) \to \Br_{\nr}(X_{\bar{k}}))
\]
of the unramified Brauer group is sufficient to control the Hasse principle on homogeneous spaces with connected or commutative stabilizers. In fact, the abelian nature of the stabilizers enables the machinery of arithmetic duality theorems {\em \`a la} Poitou--Tate. For finite noncommutative stabilizers, one expects the intervention of transcendental classes. For the closely related problem of {\em weak approximation}, the insufficiency of the algebraic Brauer--Manin obstruction was established for finite constant metabelian stabilizers, independently by Demarche--Lucchini Arteche--Neftin \cite[Corollary~5.2 and Example~5.3]{DLAN2017Grundwald}, Demeio \cite[Theorem~1.5]{Demeio2026Descent}, and Lagarde \cite[Theorem~1.5]{Lagarde2025Brauer}. On the contrary, no such example for the Hasse principle was previously known. This is the aim of the present article. Here is our main result.

\begin{thm} \label{thm:Main}
    Let $p$ be a prime number and $k$ a number field containing a primitive $p$-th root of unity. Then, there exist an integer $n \ge 1$ and a homogeneous space $X$ of $\SL_n$ over $k$ satisfying the following conditions.
    \begin{enumerate}
        \item The geometric stabilizers of $X$ are finite $p$-groups of nilpotency class $2$.

        \item The unramified algebraic Brauer group of $X$ is $\Br_{\nr,1}(X) = \Br_0(X)$. 

        \item $X(k_v) \neq \varnothing$ for all places $v$ of $k$.

        \item $X(k) = \varnothing$.
    \end{enumerate}
\end{thm}


\subsection{Organization of the manuscript} \label{subsec:Idea}

Let us explain the idea behind the construction of the homogeneous space in Theorem~\ref{thm:Main}, which is the content of Section~\ref{sec:Construction}. In \cite{BK1997Hasse, BK2001Erratum}, Borovoi and Kunyavskii introduced a general machinery to construct finite $k$-group schemes $F$ of nilpotency class $2$. This construction was studied further in \cite{Linh2024BK}, where the author established the Hasse principle and weak approximation for a subclass of homogeneous spaces with geometric stabilizers $F$. Here is the idea (see \cite[\S{3}]{Linh2024BK} for details). Let $M$ and $Z$ be finite Galois modules and consider a Galois equivariant homomorphism
\[
    \phi \colon M \otimes M \longrightarrow Z.
\]
For technical reasons, we assume $\phi$ surjective and {\em nondegenerate} in the following sense.
\begin{enumerate}
    \item If $x \in M$ such that $\phi(x \otimes y) = 0$ for all $y \in M$ then $x = 0$.

    \item  If $y \in M$ such that $\phi(x \otimes y) = 0$ for all $x \in M$ then $y = 0$.
\end{enumerate}
From $\phi$, one constructs a central extension $F$ of $M \oplus M$ by $Z$ with $Z(F) = [F,F] = Z$. For each $\beta \in \H^2(k,Z)$, the Demarche--Lucchini Arteche construction \cite[Corollaire~3.3]{DLA2019Reduction} yields a homogeneous space $X = X_\beta$ of $\SL_n$ (for some $n \ge 1$) with geometric stabilizers $F$. Furthermore, for any field extension $K \supseteq k$, one has $X(K) \neq \varnothing$ if and only if there exist $x,y \in \H^1(K,M)$ with $\beta|_K = \phi_\ast(x \cup y)$ (see Corollary~\ref{coro:RationalPointVersusCupProduct} below).

In \cite[\S{5B}]{Linh2024BK}, our choice of $(M,Z,\phi)$ is as follows. Let $A$ be a finite constant abelian group of exponent $e$ such that $k$ contains a primitive $e$-th root of unity. Let $M:=\Ind_L^k(A)$ for some finite Galois extension $L/k$, let $Z$ be the cokernel of the canonical inclusion $A \otimes A \hookrightarrow M \otimes M$, and let $\phi: M \otimes M \to Z$ be the canonical projection. Using the Poitou--Tate exact sequence in Galois cohomology, the Hasse principle was established for the obtained homogeneous spaces \cite[Th\'eor\`eme~5.6]{Linh2024BK}. Quite surprisingly, a slight modification of this construction yields the counterexample in Theorem~\ref{thm:Main}. Namely, we take $A = \Fbf_p$ and $M = M^\vee = \Fbf_p^d$ for some sufficiently large integer $d$. We now consider a different embedding
\[
    A \otimes A = \Fbf_p \longhookrightarrow M \otimes M^\vee = \Mat_{d}(\Fbf_p), \quad \lambda \longmapsto \lambda \Ibf_d,
\]       
where $\Ibf_d$ denotes the identity matrix of size $d$. Again, we take $Z$ to be the cokernel of this inclusion and $\phi$ the canonical projection. The obtained stabilizer $F$ then has order $p^{d^2 + 2d - 1}$. The smallest example produced by this construction is when $p = 2$ and $d=1+3 = 4$ over $k = \mathbb{Q}$, where $F$ has order $2^{23} = 8388608$.

Next, decompose $\Ibf_d$ into a sum of rank one matrices 
\[
    \Ibf_d = \sum_{i=1}^d E_i
\]
and then choose distinct finite places $v_1,\ldots,v_d$. Identify 
\[
    \H^2(k_{v_i}, M \otimes M^\vee) \cong M \otimes M^\vee \qquad \text{and} \qquad \H^2(k_{v_i},Z) \cong Z
\]
via the local invariants (see \eqref{eq:LocalInvariantA} below). Global class field theory then produces a class $\beta \in \H^2(k,Z)$ that vanishes outside $\{v_1,\ldots,v_d\}$ and satisfies $\beta|_{k_{v_i}} = \phi(E_i)$ for $i = 1,\ldots,d$. Since each local invariant $\beta|_{k_v}$ is represented by a matrix in $M\otimes M^\vee$ of rank $\le 1$, it is easy to produce local classes $x_v \in \H^1(k_v,M)$ and $y_v \in \H^1(k_v,M^\vee)$ with $\phi_\ast(x_v \cup y_v) = \beta|_{k_v}$, cf. Lemma~\ref{lemm:LocalPoints}. 

To show that $\beta$ is not globally of the form $\phi_\ast(x \cup y)$, consider the local invariant matrices 
\[
    A_v = (x \cup y)|_{k_v} \in \H^2(k_v, M \otimes M^\vee) \cong M \otimes M^\vee.
\]
Then $\phi(A_{v_i}) = \phi(E_i)$ for $i = 1,\ldots,d$ and $\phi(A_v) = 0$ for $v \notin \{v_1,\ldots,v_d\}$. The point here is that the dimension of the $\Fbf_p$-vector space $\H^1(k_v,\Fbf_p)$ is bounded by a constant independent of $v$. This yields a uniform bound on the rank of $A_v$ that forces 
\[
    A_{v} = \begin{cases}
        E_i & \text{if } v = v_i \quad (i = 1,\ldots,d),\\
        0 & \text{otherwise}.
    \end{cases}
\]
once $d$ is sufficiently large enough. We then derive a contradiction using the global reciprocity law (Proposition~\ref{prop:NoRationalPoints}) for the finite constant Galois module $M \otimes M^\vee$. The homogeneous space $X$ is therefore a counterexample to the Hasse principle.

In Section~\ref{sec:Obstruction}, we show that the above counterexample cannot be explained by the algebraic Brauer--Manin obstruction. The computation of the unramified algebraic Brauer group for our construction was carried out in \cite[\S{4C}]{Linh2024BK} using Demarche's formula \cite[Th\'eor\`eme~1]{Demarche2010Brauer}. This ensures $\Br_{\nr,1}(X) = \Br_0(X)$, which contributes trivially to the Brauer--Manin pairing. Now, the stabilizers of our homogeneous space are {\em supersolvable} as finite groups equipped with an outer Galois action (in fact, this outer action is {\em trivial}). Thanks to the sophisticated result of Harpaz--Wittenberg \cite[Th\'eor\`eme~B]{HW2020Galois}, the full Brauer--Manin obstruction suffices to explain the failure of the Hasse principle on $X$. This already proves the existence of a transcendental Brauer--Manin obstruction. As a matter of fact, an application of Bogomolov's formula \cite{Bogomolov1988Brauer} (see Lemma~\ref{lemm:TranscendentalBrauerGroup} below) shows that the transcendental Brauer group of $X$ is
\[
    \Br_{\nr}(X)/\Br_{0}(X) = \Br_{\nr}(X)/\Br_{\nr,1}(X) \cong \Br_{\nr}(X_{\bar{k}}) \cong \Zbf/p.
\]
To produce explicit unramified transcendental Brauer classes that give a nonvanishing Brauer--Manin condition, we start by constructing a geometrically unramified class $\bar{\Acal} \in \Br_{\nr}(X_{\bar{k}})$ using the description of the Bogomolov multiplier in Subsection~\ref{subsec:UnramifiedClass}. Next, we descend this class into a class $\Acal \in \Br(X)$ and investigate its evaluation at local points in Subsection~\ref{subsec:Evaluation}. This construction can be carried out by attaching such a class $\Acal = \Acal_L$ to {\em any} functional $L \colon M \otimes M^\vee \to \Fbf_p$. We do not prove that $\Acal \in \Br_{\nr}(X)$ algebraically; instead, we apply Harari's criterion \cite{Harari1994Fibration} to do so. A suitable choice of $L$ then gives a $\bar{\Acal} \neq 0$, that is, $\Acal$ is transcendental, cf. Subsection~\ref{subsec:TransBrauerExplicit}.

%% file: section2.tex
\section{Construction of the counterexample} \label{sec:Construction}

We fix a rational prime $p$ and a number field $k$ containing a primitive $p$-th root of unity (if $p > 2$, this ensures that $k$ is totally imaginary). Let $\Omega$ denote the set of places of $k$. We identify the finite constant Galois module $\Fbf_p = \Zbf/p$ with $\mu_p$ by fixing a generator $\zeta_p \in \mu_p$ once and for all.

In what follows, we make use of the following fact. Let $A$ be a finite constant elementary abelian $p$-group (i.e., a finite-dimensional $\Fbf_p$-vector space equipped with the trivial Galois action) and $\hat{A}:=\Hom(A,\mu_p) = \Hom(A,\Fbf_p)$ its Cartier dual (which is again a finite constant $p$-elementary abelian group). For any finite or real place $v \in \Omega$ (note that the existence of a real place $v$ forces $p = 2$), Tate's local duality asserts that the pairing
\[
    \H^2(k_v,A) \times \hat{A} \overset{\cup}{\longrightarrow} \H^2(k_v,\Fbf_p) \overset{\inv_v}{\longrightarrow} \Zbf/p
\]
is a perfect pairing of finite groups, that is, one has an isomorphism
\begin{equation} \label{eq:LocalInvariantA}
    \Inv_v^A: \H^2(k_v,A) \overset{\cong}{\longrightarrow} \Hom(\hat{A},\Zbf/p) = A
\end{equation}
which is natural in the sense that for any homomorphism $f: A \to B$ of finite constant $p$-elementary abelian groups, one has a commutative diagram
\begin{equation} \label{eq:LocalInvariantFunctorial}
    \xymatrix{
        \H^2(k_v, A) \ar[d]^{f_\ast} \ar[rr]^-{\Inv_v^{A}} && A \ar[d]^{f} \\
        \H^2(k_v, B) \ar[rr]^-{\Inv_v^B} && B.
    }
\end{equation}
We now apply Poitou--Tate duality \cite[Th\'eor\`eme~17.13]{Harari2017Cohomologie} to $A$ and $\hat{A}$. The map
\[
    \H^1(k,\hat{A}) \longhookrightarrow \prod_{v \in \Omega}\H^1(k_v,\hat{A})
\]
is injective by virtue of Chebotarev's density theorem. Dually, the map 
\[
    \H^2(k,A) \longhookrightarrow \prod_{v \in \Omega}\H^2(k_v,A)
\]
is injective as well. In view of \eqref{eq:LocalInvariantA} and the identification $\Hom(\H^0(k,\hat{A}),\Qbf/\Zbf) = A$, the last three terms of the nine-term Poitou--Tate exact sequence then yield a short exact sequence  
\begin{equation} \label{eq:ABHNForA}
    0 \longrightarrow \H^2(k,A) \overset{(\Inv^A_v)_v}{\longrightarrow} \bigoplus_{\substack{v \in \Omega \\ v \text{ finite or real}}} A \overset{\Sigma}{\longrightarrow} A \longrightarrow 0.
\end{equation}


\subsection{The numerical data} \label{subsec:ConstructionData}

Let us follow the construction in \cite[\S{3}]{Linh2024BK}.
\begin{itemize}
     \item Fix an integer $d \ge [k:\Qbf] + 3$ and let $M := \Fbf_p^{d}$ with standard basis $\{e_1,\ldots,e_{d}\}$. Denote by $\{e_1^\vee,\ldots,e_d^\vee\}$ the dual basis for $M^\vee$. Identify $M \otimes M^\vee$ with $\Mat_{d}(\Fbf_p)$ in the obvious way, that is, if $x = \sum_{i=1}^d x_i e_i$ and $y = \sum_{i=1}^d y_i e_i^\vee$ then
     \[
        x \otimes y = \sum_{i,j=1}^d x_iy_j (e_i \otimes e_j^\vee) = \begin{bmatrix} x_1 \\ \vdots \\ x_d\end{bmatrix} \begin{bmatrix} y_1 & \cdots & y_d\end{bmatrix} = \begin{bmatrix} x_1y_1 & \cdots & x_1y_d \\ \vdots & \ddots & \vdots\\ x_d y_1 & \cdots & x_d y_d\end{bmatrix}.
    \]

    \item Let $Z:=\Mat_{d}(\Fbf_p)/\langle \Ibf_{d} \rangle$ and let $\phi \colon M \otimes M^\vee \twoheadrightarrow Z$ be the canonical projection.

    \item Define the finite constant group scheme $F$ as an extension of $M \oplus M^\vee$ by $Z$ as follows. Put $F := Z \times (M \oplus M^\vee)$ as a set, and then equip it with the group law given by
    \begin{equation} \label{eq:GroupLawOnF}
        (z,x,y)(z',x',y'):=(z + z' + \phi(x \otimes y'), x+x', y+y')
    \end{equation}
    for all $x,x' \in M$, $y,y' \in M^\vee$, and $z,z' \in Z$.
\end{itemize}
Hence, $F$ is a $p$-group of nilpotency class $2$ and order $p^{d^2+2d-1}$.

\begin{lemm} \label{lemm:CenterAndDerivedSubgroup}
    The center and the commutator subgroup of $F$ both equal $Z$.
\end{lemm}
\begin{proof}
    This is basically \cite[Lemmes~3.5 and ~3.6]{Linh2024BK}. We recall the proof here for the reader's convenience. On one hand, $[F,F] \subseteq Z$ because $F/Z = M \oplus M^\vee$ is abelian. On the other hand, every generating element $\phi(x \otimes y) \in Z$ is a commutator in $F$, namely
    \begin{equation} \label{eq:CenterAndDerivedSubgroup}
        (\phi(x \otimes y),0,0) = (0,x,0)(0,0,y)(0,x,0)^{-1}(0,0,y)^{-1}
    \end{equation}
    by virtue of \eqref{eq:GroupLawOnF}. Thus $Z \subseteq [F,F]$.

    Now assume $f_0 = (z_0,x_0,y_0) \in F$ is a central element. For all $y \in M^\vee$ in view of
    \eqref{eq:GroupLawOnF}, the condition $(z_0,x_0,y_0) (0,0,y) = (0,0,y) (z_0,x_0,y_0)$ implies $\phi(x_0 \otimes y) = 0$, or
    \[
        x_0 \otimes y = \lambda \Ibf_{d}
    \]
    for some $\lambda \in \Fbf_p$ {\em a priori} depending on $y$. Since the left hand side has rank at most one, we necessarily have $\lambda = 0$ and hence $x_0 \otimes y = 0$ for all $y \in M$. This forces $x_0 = 0$. Similarly, $y_0 = 0$, so $f_0 = (z_0,0,0) \in Z$. Conversely, $Z$ is central in $F$ by construction, so the center of $F$ is $Z$.
\end{proof}

The next step is to decompose the identity matrix $\Ibf_{d}$ into rank one matrices, say
\[
    \Ibf_d = \sum_{i=1}^d e_i \otimes e_i^\vee.
\]
In particular, one has 
\begin{equation} \label{eq:SumOfPhiEi}
    \sum_{i=1}^{d} \phi(e_i \otimes e_i^\vee) = \phi(\Ibf_{d}) = 0.
\end{equation}


\subsection{Existence of local points} \label{subsec:ConstructionLocalPoints}

Let us now construct our homogeneous spaces using the machinery of nonabelian second Galois cohomology. Keep the notation from the previous subsection. For an exposition on the subject, the reader is invited to see \cite[\SS{1 and 2}]{Borovoi1993H2}, \cite[\S{1}]{FSS1998H2}, and \cite[\S{2.2}]{DLA2019Reduction}.

We regard the group $F$ constructed in the previous subsection as a finite constant $k$-group scheme, that is, the absolute Galois group $\Gal(\bar{k}/k)$ acts trivially on $F = F(\bar{k})$. Let us consider the set $\H^2(k,F)$, which can be defined in terms of cocycles, extensions, or {\em gerbes}. This set has a distinguished element $\nbf_F$ (which corresponds to the trivial Galois action on $F$) and also a subset of {\em neutral} elements (which correspond to the Galois actions on $F$ that coincide with the trivial action modulo $F$-conjugation), cf. \cite[\S{1.4} and Lemma~2.2]{Borovoi1993H2} or \cite[Definition~1.25]{FSS1998H2}. In particular, the set $\H^2(k, F)$ is nonempty.

Since the center of $F$ is $Z$ by Lemma~\ref{lemm:CenterAndDerivedSubgroup}, it follows from \cite[\S{1.6} and Lemma~1.7]{Borovoi1993H2} that there is a simply transitive action
\[
    \H^2(k,Z) \times \H^2(k,F) \longrightarrow \H^2(k,F), \qquad (\beta,\eta) \longmapsto \beta + \eta.
\]
Therefore, any element $\eta \in \H^2(k,F)$ can be written as $\eta = \beta + \nbf_F$ for some uniquely determined class $\beta \in \H^2(k,Z)$. The construction of Demarche and Lucchini Arteche \cite[Corollaire~3.3]{DLA2019Reduction} yields a homogeneous space $X = X_\eta = X_\beta$ of $\SL_n$ (for some integer $n \ge 1$) with geometric stabilizers $F$ and {\em Springer class} $\eta_X = \eta$. This class is neutral if and only if there exists an $\SL_n$-equivariant morphism $T \to X$, where $T$ is a principal homogeneous space of $\SL_n$. Since $\H^1(k,\SL_n) = 1$ by Speiser's version of Hilbert's Theorem 90 \cite[Chapitre~X, Corollaire to Proposition~3]{Serre2004Locaux}, we have $T(k) \neq \varnothing$. Therefore, $\eta$ is neutral precisely when $X(k) \neq \varnothing$.

According to \cite[Proposition~2.3]{Borovoi1993H2}, the class $\eta = \beta + \nbf_F$ is neutral if and only if $\beta$ lies in the image of the connecting map $\Delta \colon \H^1(k, M \oplus M) \to \H^2(k, Z)$ associated with the central extension 
\[
    1 \longrightarrow Z \longrightarrow F \longrightarrow M \oplus M^\vee \longrightarrow 1,
\]
see e.g. \cite[Chapitre~VII, Annexe, Proposition~2]{Serre2004Locaux}. The following computation of the map $\Delta$ is given in \cite[Lemme~3.8]{Linh2024BK}; we recall it here for the reader's convenience.

\begin{lemm} \label{lemm:ConnectingMap}
    Under the identification $\H^1(k, M \oplus M^\vee) = \H^1(k,M) \oplus \H^1(k,M^\vee)$, the connecting map 
    \[
        \Delta \colon \H^1(k,M) \oplus \H^1(k,M^\vee) \longrightarrow \H^2(k, Z)
    \]
    is given by $\Delta(x,y) = \phi_\ast(x \cup y)$, where $x \cup y \in \H^2(k, M \otimes M^\vee)$.
\end{lemm}
\begin{proof}
    Let $x \colon \Gal(\bar{k}/k) \to M$ and $y \colon \Gal(\bar{k}/k) \to M^\vee$ be $1$-cocycles (or equivalently, continuous homomorphisms, because $M$ and $M^\vee$ are finite constant $k$-groups). We lift the $(M \oplus M)$-valued cocycle $(x,y)$ to the $F$-valued cochain $f = (0,x,y)$. For all $s,t \in \Gal(\bar{k}/k)$, using \eqref{eq:GroupLawOnF}, one has
    \[
        f_s f_t = (0,x_s,y_s)(0,x_t,y_t) = (\phi(x_s \otimes y_t), x_s + x_t, y_s + y_t) = (\phi(x_s \otimes y_t), x_{st},y_{st}) = (\phi(x_s \otimes y_t), 0,0) f_{st}.
    \]
    By the standard formula for the connecting map, the class $\Delta(x,y)$ is represented by the $2$-cocycle
    \[
        \Gal(\bar{k}/k) \times \Gal(\bar{k}/k) \longrightarrow Z, \qquad (s,t) \longmapsto \phi(x_s \otimes y_t).
    \]
    On the other hand, since $\Gal(\bar{k}/k)$ acts trivially on $M$ and $M^\vee$, the cup-product $x \cup y$ is
    \[
        \Gal(\bar{k}/k) \times \Gal(\bar{k}/k) \longrightarrow M \otimes M^\vee, \qquad (s,t) \longmapsto x_s \otimes \tensor[^s]{y}{_t} = x_s \otimes y_t.
    \]
    This proves the lemma.
\end{proof}

The above discussion remains valid when the field $k$ is replaced by any extension. Therefore, Lemma~\ref{lemm:ConnectingMap} gives the following cohomological criterion for the existence of rational points.

\begin{coro} \label{coro:RationalPointVersusCupProduct}
    Given $\beta \in \H^2(k,Z)$, let $X = X_\beta$ be the homogeneous space constructed from the class $\beta + \nbf_F \in \H^2(k,F)$ using the Demarche--Lucchini Arteche construction mentioned above. Then, for any field extension $K \supseteq k$, one has $X(K) \neq \varnothing$ if and only if there exist $x,y \in \H^1(K,M)$ with $\beta|_K = \phi_\ast(x \cup y)$ in $\H^2(K,Z)$.
\end{coro}

In order to construct our counterexamples to the Hasse principle, we proceed to carefully choose a class $\beta \in \H^2(k,Z)$. Take $d$ pairwise distinct places $v_1,\ldots,v_d \in \Omega$, none of which is complex nor lies over the rational prime $p$. Recall from \eqref{eq:LocalInvariantA} that we have an isomorphism 
\[
    \Inv_{v_i}^Z \colon \H^2(k_{v_i},Z) \overset{\cong}{\longrightarrow} Z
\]
for $i=1,\ldots,d$. In view of \eqref{eq:ABHNForA} and \eqref{eq:SumOfPhiEi}, there exists $\beta \in \H^2(k,Z)$ such that $\beta|_{k_v} = 0 \in \H^2(k_v,Z)$ for $v \neq v_1,\ldots,v_d$ and
\[
    \Inv_{v_i}^Z(\beta) = \phi(e_i \otimes e_i), \quad i = 1,\ldots,d.
\]

\begin{lemm} \label{lemm:LocalPoints}
    We have $\prod_{v \in \Omega} X(k_v) \neq \varnothing$.
\end{lemm}
\begin{proof}
    Let $v \in \Omega$. In view of Corollary~\ref{coro:RationalPointVersusCupProduct}, we need to construct cohomology classes $x_v \in \H^1(k_v,M)$ and $y_v \in \H^1(k_v,M^\vee)$ with $\beta|_{k_v} = \phi_\ast(x_v \cup y_v)$. If $v \neq v_1,\ldots,v_d$ then $\beta|_{k_v} = 0$, so we can simply take $x_v = 0$ and  $y_v = 0$. Therefore, we assume $v = v_i$ for some $i = 1,\ldots,d$, so that
    \[
        \Inv_{v_i}^Z(\beta) = \phi(e_i \otimes e_i^\vee),
    \]
    By Tate's local duality, the pairing
    \[
        \H^1(k_{v_i}, \Fbf_p) \times \H^1(k_{v_i}, \Fbf_p) \overset{\cup}{\longrightarrow} \H^2(k_{v_i}, \Fbf_p) \cong \tensor[_p]{\Br}{}(k_{v_i}) \overset{\inv_{v_i}}{\longrightarrow} \Zbf/p \subseteq \Qbf/\Zbf
    \]
    is a perfect pairing of finite groups (this is still true for real $v_i$ and $p = 2$). Also note that $\H^1(k_{v_i},\Fbf_p) \neq 0$ (this follows from \cite[Chapter~VII, Proposition~6.8]{Milne2020CFT}). In particular, there are classes $a_i,b_i \in \H^1(k_{v_i},\Fbf_p)$ such that
    \[
        \inv_{v_i}(a_i \cup b_i) = 1 \in \Zbf/p.
    \]
    Then, under the identifications $\H^1(k_v,M) = \H^1(k_v,M^\vee) = \H^1(k_v,\Fbf_p)^d$, the classes
    \[
        x_{v_i}:=e_i \cup a_i = (0,\ldots,a_i,\ldots,0) \qquad \text{and} \qquad y_{v_i}:=e_i \cup b_i = (0,\ldots,b_i,\ldots,0) 
    \]    
    has invariant
    \[
        \Inv_{v_i}^{M \otimes M^\vee}(x_{v_i} \cup y_{v_i}) = e_i \otimes e_i
    \]
    (recall that the isomorphism $\Inv_{v_i}^{M \otimes M^\vee} \colon \H^2(k_{v_i}, M \otimes M^
    vee) \overset{\cong}{\rightarrow} M \otimes M^\vee$ was defined from \eqref{eq:LocalInvariantA}). It follows from functoriality of the mentioned isomorphism, i.e. commutativity of diagram \eqref{eq:LocalInvariantFunctorial}, that
        \begin{equation*}
            \Inv_{v_i}^Z(\phi_\ast(x_{v_i} \cup y_{v_i})) = \phi(\Inv_{v_i}^{M \otimes M^\vee}(x_{v_i} \cup y_{v_i})) = \phi(e_i \otimes e_i) = \Inv_{v_i}^Z(\beta),
        \end{equation*}
    which implies $\beta|_{k_{v_i}} = \phi_\ast(x_{v_i} \cup y_{v_i})$. This proves the lemma.
\end{proof}


\subsection{Absence of rational points} \label{subsec:ConstructionRationalPoints}

We now prove that the constructed homogeneous space $X$ has no $k$-rational point. The idea is to exploit the uniform bound on the $p$-ranks of the local cohomology groups $\H^1(k_v,\Fbf_p)$.

\begin{prop} \label{prop:NoRationalPoints}
    The homogeneous space $X$ constructed from Subsection~\ref{subsec:ConstructionLocalPoints} is a counterexample to the Hasse principle, that is, $X(k) = \varnothing$.
\end{prop}
\begin{proof}
    In view of Corollary~\ref{coro:RationalPointVersusCupProduct}, it suffices to show that $\phi_\ast(x \cup y) \neq \beta$ for all $x,y \in \H^1(k,M)$.

    Since $k$ contains $\mu_p$, it follows from \cite[Chapter~VII, Proposition~6.8]{Milne2020CFT} that
    \[
        \# \H^1(k_v,\Fbf_p) = \# (k_v^\times / k_v^{\times p}) = p^2 |p|_{k_v}^{-1}
    \]
    for any place $v \in \Omega$. Put $r_v:=\dim_{\Fbf_p} \H^1(k_v,\Fbf_p)$. If $v$ lies over $p$ then 
    \[
        \# \H^1(k_v,\Fbf_p) \le p^2 p^{[k_v:\Qbf_p]} \le p^{2 + [k_v:\Qbf_p]}
    \]
    Since $d \ge [k:\Qbf] + 3$ by construction, one has $r_v < d$. If $v$ does not lie over $p$ then $r_v \le 2$.

    For each finite place $v \in \Omega$, consider the matrix
    \[
        A_v:=\Inv_v^{M \otimes M^\vee}(x \cup y).
    \]
    Explicitly, if we identify $\H^1(k_v,M)$ and $\H^1(k_v,M^\vee)$ with $\H^1(k_v,\Fbf_p)^{d}$ and write $x = (x_1,\ldots,x_{d})$ and $y = (y_1,\ldots,y_{d})$, then $A_v = [\inv_v(x_i \cup y_j)]_{i,j=1}^{d}$. By choosing an $\Fbf_p$-basis $\{w_1,\ldots,w_{r_v}\}$ for $\H^1(k_v,\Fbf_p)$, one sees that $A_v$ decomposes as
    \[
        A_v = B_v^t [\inv_v(w_i \cup w_j)]_{i,j=1}^{r_v} C_v,
    \]
    where the columns of the respective matrices $B_v,C_v \in \Mat_{r_v \times d}(\Fbf_p)$ are the coordinates of $x_i$ and $y_i$ in the basis $\{w_1,\ldots,w_{r_v}\}$. In particular, the rank of $A_v$ is at most $r_v$. 
    
    Now assume for the sake of contradiction that $\beta = \phi_\ast(x \cup y)$. Consider the following cases.
    
    \begin{enumerate}
        \item If $v = v_i$ for some $i=1,\ldots,d$ then $r_v \le 2$ because $v_i \nmid p$ by choice. Since
        \[
            \phi(A_{v_i}) = \phi(\Inv_{v_i}^{M \otimes M^\vee}(x \cup y)) = \Inv_{v_i}^Z(\phi_\ast(x \cup y)) = \Inv_{v_i}^Z(\beta) = \phi(e_i \otimes e_i),
        \]
       we have $A_{v_i} = e_i \otimes e_i + \lambda \Ibf_{d}$ for some $\lambda \in \Fbf_p$. If $\lambda \neq 0$, then $A_{v_i}$ is diagonal with $(d-1)$ entries $\lambda$, so it has rank $\ge d - 1 > 2$, which is impossible. Hence $\lambda = 0$ and thus $A_{v_i} = e_i \otimes e_i$.

        \item If $v \neq v_1,\ldots,v_d$, then the same calculation leads to $\phi(A_{v}) = \Inv_{v}^Z(\beta) = 0$, so $A_{v} = \lambda \Ibf_{d}$ for some $\lambda \in \Fbf_p$. But $A_v$ has rank $\le r_v < d$, so this forces $\lambda = 0$, i.e. $A_v = 0$.
    \end{enumerate}
    In view of \eqref{eq:SumOfPhiEi}, one has
        \begin{equation*}
            \sum_{\substack{v \in \Omega \\ v \text{ finite or real}}} \Inv_v^{M \otimes M^\vee} (x \cup y) = \sum_{i=1}^d A_{v_i} = \sum_{i=1}^d (e_i \otimes e_i) = \Ibf_{d} \neq 0.
        \end{equation*}
    But this contradicts the exactness of \eqref{eq:ABHNForA}. This contradiction shows that $X(k) = \varnothing$.
\end{proof}

%% file: section3.tex
\section{Transcendental Brauer--Manin obstruction} \label{sec:Obstruction}

We now show that the failure of the Hasse principle for the homogeneous space $X$ from Section~\ref{sec:Construction} cannot be explained by the algebraic part of the Brauer--Manin obstruction. 

For a finite $\Gal(\bar{k}/k)$-module $N$, put
\[
    \Sha_{\omega}^{1}(k,N) :=\left\{\alpha\in \H^{1}(k,N) : \alpha|_{k_v}=0 \text{ for all but finitely many } v\in\Omega \right\}.
\]    
We recall the following result from \cite[Proposition~4.6]{Linh2024BK}.

\begin{lemm} \label{lemm:AlgebraicBrauerGroup}
    Let $M$ and $Z$ be finite $\Gal(\bar{k}/k)$--modules and let
    \[
        \phi \colon M\otimes M\longrightarrow Z
    \]
    be a $\Gal(\bar{k}/k)$-equivariant homomorphism. Let $F := Z \times (M \oplus M)$ as a set, equipped with the group law given by \eqref{eq:GroupLawOnF}, and endowed with the coordinatewise action of $\Gal(\bar{k}/k)$. Then 
    \[
        1 \longrightarrow Z \longrightarrow F \longrightarrow M \oplus M \longrightarrow 1
    \]
    is a central extension of finite $k$-group schemes. Assume that $Z$ is equal to the center and also the derived subgroup of $F$. Then, for any homogeneous space $X = X_\eta$ of $\SL_n$ constructed from a class $\eta \in \H^2(k,F)$ by the Demarche--Lucchini Arteche construction (cf. Subsection~\ref{subsec:ConstructionLocalPoints}), one has
    \[
        \Br_{\nr,1}(X)/\Br_0(X) \cong \Sha^1_{\omega}\bigl(k, \widehat{M}\bigr)^{\oplus 2},
    \]
    where $\widehat{M}:=\Hom(M, \bar{k}^\times)$ denotes the Cartier dual of $M$.
\end{lemm}

In the situation where $M = M^\vee$, $Z$, and $\phi$ are chosen as in Section~\ref{sec:Construction}, the Galois action on $F$ ({\em a fortiori} on $M = M^\vee = \Fbf_p^d$ and on $Z$) is trivial. Moreover, since $k$ contains a $p$-th root of unity, the Galois module
\[
    \widehat{M} = \Hom(M,\mu_p) \cong \Fbf_p^d,
\]
is also constant. Lemma~\ref{lemm:AlgebraicBrauerGroup} therefore gives
\[
    \Br_{\nr,1}(X)/\Br_0(X) \cong \Sha^1_{\omega}(k, \Fbf_p)^{\oplus (2d)},
\]
The group on the right hand side vanishes by the Chebotarev density theorem, because any nonzero character
$\chi \in \H^1(k,\Fbf_p) = \Hom(\Gal(\bar{k}/k),\Fbf_p)$ remains nonzero after localization at infinitely many unramified places. It follows that the unramified algebraic group of $X$ is 
\[
    \Br_{\nr,1}(X) = \Br_0(X),
\]
so it contributes trivially to the Brauer--Manin pairing. This concludes the proof of Theorem~\ref{thm:Main} on the insufficiency of the algebraic Brauer--Manin obstruction on $X$.

As explained at the end of Subsection~\ref{subsec:Idea}, since $F$ is nilpotent constant and hence supersolvable, \cite[ Th\'eor\`eme~B]{HW2020Galois} implies the existence of a transcendental Brauer--Manin obstruction on $X$. In particular, we have $\Br_{\nr}(X) \supsetneq \Br_0(X)$, so the transcendental Brauer group
\[
    \Br_{\nr}(X)/\Br_{0}(X) = \Br_{\nr}(X)/\Br_{\nr,1}(X) \longhookrightarrow \Br_{\nr}(X_{\bar{k}})
\]
embeds as a {\em nontrivial} subgroup of $\Br_{\nr}(X_{\bar{k}})$.

\begin{lemm} \label{lemm:TranscendentalBrauerGroup}
    We have $\Br_{\nr}(X)/\Br_{0}(X) \cong \Br_{\nr}(X_{\bar{k}}) \cong \Zbf/p$.   
\end{lemm}
\begin{proof}
    According to \cite[Proposition~4.9]{Linh2024BK}, the geometric Brauer group $\Br_{\nr}(X_{\bar{k}})$ (that is, the Bogomolov multiplier $\B_0(F)$ of $F$) is given by
    \[
        \Br_{\nr}(X_{\bar{k}}) = \B_0(F) = \Hom(\Ker(\phi)/H, \Qbf/\Zbf),
    \]
    where $H \subseteq M \otimes M^\vee = \Mat_d(\Fbf_p)$ is the subgroup generated by rank one matrices $x \otimes y$ (where $x \in M$ and $y \in M^\vee = M$) with $\phi(x \otimes y) = 0$. This condition is equivalent to
    \[
        x \otimes y = \lambda \Ibf_d
    \]
    for some $\lambda \in \Fbf_p$. As in the proof of Lemma~\ref{lemm:CenterAndDerivedSubgroup}, since the left hand side has rank at most one, this forces $x \otimes y = 0$. Therefore, $H = 0$ and thus
    \[
        \Br_{\nr}(X_{\bar{k}}) = \Hom(\Ker(\phi),\Qbf/\Zbf) \cong \Hom(\Fbf_p, \Qbf/\Zbf) = \Zbf/p.
    \]
    Since the only subgroups of $\Zbf/p$ are $\{0\}$ and $\Zbf/p$ itself, and since $\Br_{\nr}(X)/\Br_{0}(X)$ has been shown to be nontrivial, the claimed isomorphisms $\Br_{\nr}(X)/\Br_{0}(X) \cong \Br_{\nr}(X_{\bar{k}}) \cong \Zbf/p$ follows.
\end{proof}


\subsection{Producing geometrically unramified classes} \label{subsec:UnramifiedClass}

Our next task is to describe an explicit transcendental Brauer class $\Acal \in \Br_{\nr}(X)$ which obstructs the Brauer--Manin condition. First, let us recall the following description of $\Br(X)$, which is valid for any homogeneous space of a simply connected semisimple group with finite geometric stabilizers. Choose a geometric point $\bar{P} \in X(\bar{k})$. Then $\pi_1^{\text{\'et}}(X_{\bar{k}}, \bar{P}) = F$ and $\Pi_X:=\pi_1^{\text{\'et}}(X,\bar{P})$ fits in an exact sequence 
\[
    1 \longrightarrow F \longrightarrow \Pi_X \longrightarrow \Gal(\bar{k}/k) \longrightarrow 1,
\]
cf. \cite[Expos\'e~IX, Th\'eor\`eme~6.1]{Grothendieck71}. Furthermore, by construction, the above extension induces the trivial outer action of $\Gal(\bar{k}/k)$ on $F$. More explicitly, according to \cite[p.~1577]{DLA2019Reduction}, let
\[
    b \colon \Gal(\bar{k}/k) \times \Gal(\bar{k}/k) \longrightarrow Z, \qquad (s,t) \longmapsto b_{s,t}
\]
be a {\em normalized} $2$-cocycle representing $\beta$. Then $\Pi_X = F \times \Gal(\bar{k}/k)$ as compact Hausdorff totally disconnected topological spaces, with the group law given by
\begin{equation} \label{eq:GroupLawOnPiX}
    (f,s)(f',t):=(ff' b_{s,t}^{-1}, st).
\end{equation}
According to \cite[Proposition~3.2]{LA2019Brauer}, the full Brauer group of $X$ is
\[
    \Br(X) \cong \H^2(\Pi_X, \Qbf/\Zbf(1)),
\]
where the \'etale fundamental group $\Pi_X$ acts on the Galois module $\Qbf/\Zbf(1):=\varinjlim_n \mu_n$ via its quotient $\Gal(\bar{k}/k)$. An element of $\Br(X)$ therefore can be represented by a cocycle
\[
    \Pi_X \times \Pi_X \longrightarrow \Qbf/\Zbf(1).
\]
But since $\Pi_X$ acts trivially on the subgroup $\mu_p \cong \Fbf_p$ of $\Qbf/\Zbf(1)$, and we already know that $\Br_{\nr}(X)/
\Br_0(X) \cong \Zbf/p$, we shall look for unramified Brauer classes represented by cocycles
\[
    \Acal\colon \Pi_X \times \Pi_X \longrightarrow \Fbf_p.
\]
The idea is as follows. The class $[\Acal]$ must remain unramified when restricted to 
\[
    \Br(X_{\bar{k}}) \cong \H^2(F,\Qbf/\Zbf),
\]
that is, pulling $[\Acal]$ back to $F$ yields an element of the Bogomolov multiplier $\B_0(F)$. So we first focus on this restriction. By the proof of \cite[Lemma~5.1]{Bogomolov1988Brauer}, any element of $\B_0(F)$ is inflated from the abelianization of $F$, i.e., the group $F/Z = M \oplus M^\vee$. Let us now follow the proof of \cite[Lemme~4.8 and Proposition~4.9]{Linh2024BK} very closely. We have a well-known isomorphism
\[
    \H^2(M \oplus M^\vee, \Qbf/\Zbf) \overset{\cong}{\longrightarrow} \Hom\left(\bigwedge^2(M \oplus M^\vee), \Qbf/\Zbf\right).
\]
Any bilinear map $\Psi\colon (M \oplus M^\vee) \times (M \oplus M^\vee) \to \Qbf/\Zbf$ is a normalized $2$-cocycle, and the above isomorphism associates its class to the ``commutator map'', which can be explicitly computed as
\[
    \tilde{\Psi} \colon \bigwedge^2(M \oplus M^\vee) \longrightarrow \Qbf/\Zbf, \qquad a \wedge a' \longmapsto \Psi(a,a') - \Psi(a',a).
\]
On the other hand, the fact that $F$ is of nilpotency class $2$ with center $Z$ and abelianization $M \oplus M^\vee$ gives us another (surjective) commutator map
\begin{equation} \label{eq:CommutatorMapLambda}
    \lambda\colon \bigwedge^2(M \oplus M^\vee) \longrightarrow Z, \qquad (x,y) \wedge (x', y') \longmapsto \phi(x \otimes y') - \phi(x' \otimes y).
\end{equation}
Under the usual identification
\begin{equation} \label{eq:IdentificationOfLambda2M}
    \bigwedge^2(M \oplus M^\vee) = \bigwedge^2 M \oplus (M \otimes M^\vee) \oplus \bigwedge^2 M^\vee
\end{equation}
given by
\[
    (x,y) \wedge (x',y') = (x \wedge x', x \otimes y' - x' \otimes y, y \wedge y'),
\]
the kernel of \eqref{eq:CommutatorMapLambda} is
\[
    \Ker(\lambda) = \bigwedge^2 M \oplus \Ker(\phi) \oplus \bigwedge^2 M^\vee,
\]
and its subgroup generated by simple wedges is 
\[
    S_\lambda = \bigwedge^2 M \oplus \{0\} \oplus \bigwedge^2 M^\vee.
\]
Then, the inflation to $F$ of $[\Psi] \in \H^2(M \oplus M^\vee, \Qbf/\Zbf)$ lands in $\B_0(F)$ precisely when $\tilde{\Psi}|_{S_\lambda} = 0$.

One way to produce such a bilinear map $\Psi$ is to consider
\[
    \Psi_L\colon (M \oplus M^\vee) \times (M \oplus M^\vee) \longrightarrow \Fbf_p, \qquad ((x,y),(x',y')) \longmapsto L(x \otimes y').
\]
associated with any linear form $L \colon M \otimes M^\vee \to \Fbf_p$. Under the identification \eqref{eq:IdentificationOfLambda2M}, we have
\begin{align*}
    \tilde{\Psi}_L (x \wedge x', 0 ,0) & =  \tilde{\Psi}_L((x, 0) \wedge (x',0)) = \Psi_L((x,0),(x',0)) - \Psi_L((x',0),(x,0)) \\
    & = L(x \otimes 0 - 0 \otimes x') = 0.
\end{align*}
for all $x,x' \in M$. Similarly, $\tilde{\Psi}_L(0,0,y \wedge y') = 0$ for all $y,y' \in M$, so $\tilde{\Psi}_L|_{S_\lambda} = 0$. Therefore, inflating $\Psi_L$ to $F$ gives a 2-cocycle
\[
    \bar{\Acal}_L \colon F \times F \longrightarrow \Fbf_p
\]
that represents an unramified class in $\H^2(F,\Qbf/\Zbf) = \Br(X_{\bar{k}})$.


\subsection{Evaluation at local points} \label{subsec:Evaluation}

The next step is to lift the geometric class $[\bar{\Acal}_L]$ to a class in $\Br(X) \cong \H^2(\Pi_X, \Qbf/\Zbf(1))$, which will be shown to be unramified in Subsection~\ref{subsec:TransBrauerExplicit}. It is represented by a cocycle
\[
    \Acal_L \colon \Pi_X \times \Pi_X \longrightarrow \Fbf_p \cong \mu_p.
\]
To this end, we make use of the fact that $\bar{\Acal}_L$ is inflated from the bilinear form 
\[
    \Psi_L\colon (M \oplus M^\vee) \times (M \oplus M^\vee) \longrightarrow  \Fbf_p, \qquad ((x,y),(x',y')) \longmapsto L(x \otimes y').
\]
    
\begin{lemm} \label{lemm:TheHomomorphismQ}
    The projection $F \to M \oplus M^\vee$ lifts to a continuous homomorphism $q \colon \Pi_X \to M \oplus M$.
\end{lemm}
\begin{proof}
    We identify $\Pi_X$ to $F \times \Gal(\bar{k}/k)$ set-theoretically with the group law given by \eqref{eq:GroupLawOnPiX}. Also recall that $F = Z \times (M \oplus M)$ with the group law given by \eqref{eq:GroupLawOnF}. The map
    \[
        q \colon \Pi_X \longrightarrow M \oplus M^\vee, \qquad ((z,x,y),s) \longmapsto (x,y)
    \]
    is a group homomorphism, because 
        \begin{align*}
            ((z,x,y),s)((z',x',y'),t) & = ((z,x,y)(z',x',y')(-b_{s,t},0,0),st) \\
            & = ((z+z'+\phi(x \otimes y')-b_{s,t},x+x',y+y',st)
        \end{align*}
    for all $x,x' \in M$, $y,y' \in M^\vee$, $z,z' \in Z$, and $s,t \in \Gal(\bar{k}/k)$.
\end{proof}

Let $\Acal_L$ be the $2$-cocycle inflated from $\Psi_L$ along the homomorphism $q$ from Lemma~\ref{lemm:TheHomomorphismQ}. Before proving that $[\Acal_L] \in \H^2(\Pi_X,\Fbf_p)$ yields an unramified class in $\H^2(\Pi_X,\Qbf/\Zbf(1)) \cong \Br(X)$, we study its evaluation at the local points of $X$.

\begin{prop} \label{prop:EvaluationOfAL}
    Let $X$ be the homogeneous space of $\SL_n$ constructed from the class $\beta \in \H^2(k,Z)$ as in the beginning of Subsection~\ref{subsec:ConstructionLocalPoints}. Let $K \supseteq k$ be a field extension and assume 
    \[
        a = (x,y) \colon \Gal(\bar{K}/K) \longrightarrow M \oplus M^\vee = F/Z
    \] 
    is a cocycle such that $\beta|_K = \phi_\ast[x \cup y]$ in  $\H^2(K,Z)$. Let $\tensor[_a]{F}{}$ be the inner $K$-form of $F$ twisted by $a$, which fits in an extension
    \begin{equation} \label{eq:TwistedFormAF}
        0 \to Z \to \tensor[_a]{F}{} \to M \oplus M^\vee \to 0
    \end{equation}
    of finite $K$-group schemes. Then the following are true.

    \begin{enumerate}
        \item \label{prop:EvaluationOfAL:1} $X_K$ is isomorphic to $\tensor[_a]{F}{} \backslash \SL_{n,K}$ (in particular, $X(K) \neq \varnothing$).

        \item \label{prop:EvaluationOfAL:2} Let $X(K) \to \H^1(K,\tensor[_a]{F}{})$ be the evaluation map associated with the $\tensor[_a]{F}{}$-torsor $\SL_{n,K} \to X_K$. Let $P \in X(K)$, let $\fbf: \Gal(\bar{K}/K) \to \tensor[_a]{F}{}$ be a cocycle representing its evaluation, and let $a' = (x',y') \colon \Gal(\bar{K}/K) \to M \oplus M^{\vee}$ be its pushforward along the projection $\tensor[_a]{F}{} \to M \oplus M^\vee$. Then, we have
        \begin{equation*}
            \beta|_K = \phi_\ast[(x+x') \cup (y+y')] \in \H^2(K,Z).
        \end{equation*}
        On the other hand, the value of the above class $[\Acal_L] \in \Br(X)$ when evaluated at $P$ is
        \begin{equation*}
            [\Acal_L](P) = L_\ast[(x + x') \cup (y + y')] \in \H^2(K,\Fbf_p) = \tensor[_p]{\Br}{}(K).
        \end{equation*}
    \end{enumerate}
\end{prop}
\begin{proof}
   Claim~\ref{prop:EvaluationOfAL:1} is \cite[Proposition~3.3]{Linh2024BK}. We prove the claim about $\beta|_K$ in~\ref{prop:EvaluationOfAL:2}. By \cite[Lemme~3.8]{Linh2024BK}, the connecting map $\Delta\colon \H^1(K,M \oplus M) \to \H^2(K,Z)$ associated with the extension
    \[
        0 \longrightarrow Z \longrightarrow F \longrightarrow M \oplus M \longrightarrow 0
    \]
    takes $[a]$ to $\phi_\ast[x \cup y]$ and $[a + a']$ to $\phi_\ast[(x + x') \cup (y + y')]$. On the other hand, the connecting map $\tensor[_a]{\Delta}{}: \H^1(K,M \oplus M) \to \H^2(K,Z)$ associated with \eqref{eq:TwistedFormAF}, takes $[a']$ to $0$ because $[a']$ comes from $[\fbf] \in \H^1(K,\tensor[_a]{F}{})$. By the ``twisting formula'' for connecting maps  \cite[Chapitre~I, Proposition~44]{Serre1994Galois}, one has
    \[
        \phi_\ast[(x + x') \cup (y + y')] = \Delta([a + a']) = \tensor[_a]{\Delta}{}([a']) + \Delta([a]) = \phi_\ast[x \cup y] = \beta|_K.
    \]
    We now compute the evaluation $[\Acal_L](P)$ by a direct cocycle computation. First of all, the choice of an embedding $\bar{k} \hookrightarrow \bar{K}$ extending $k \hookrightarrow K$ yields a homomorphism $\pi\colon \Gal(\bar{K}/K) \to \Gal(\bar{k}/k)$. The condition $\beta|_K = \phi_\ast[x \cup y]$ can be rewritten as
    \[
        b_{\pi(s),\pi(t)} = \phi(x_s \otimes y_t)c_sc_tc_{st}^{-1}
    \]
    for some cochain $c \colon \Gal(\bar{K}/K) \to Z$. Put $\tilde{a}_s:=(0,a_s) \in F$, then $\phi(x_s \otimes y_t) = \tilde{a}_s \tilde{a}_t \tilde{a}_{st}^{-1}$, so that
    \[
        b_{\pi(s),\pi(t)} = (\tilde{a}_s c_s) (\tilde{a}_t c_t) (\tilde{a}_{st} c_{st})^{-1}
    \]
    for all $s,t \in \Gal(\bar{K}/K)$, see \cite[p.~366]{Linh2024BK} (note that $F$ is finite constant and $Z$ is central in $F$). Upon replacing $\tilde{a}_s$ by $\tilde{a}_s c_s$, we may assume
    \[
        b_{\pi(s),\pi(t)} = \tilde{a}_s \tilde{a}_t \tilde{a}_{st}^{-1} = \tilde{a}_{st}^{-1} \tilde{a}_s \tilde{a}_t.
    \]
    Since the action of $\Gal(\bar{K}/K)$ on $\tensor[_a]{F}{}$ is given by $s \cdot f= \tilde{a}_s f \tilde{a}_s^{-1}$, one checks that the map
    \[
        \omega \colon \tensor[_a]{F}{} \rtimes \Gal(\bar{K}/K) \longrightarrow \Pi_X, \qquad (f,s) \longmapsto (f \tilde{a}_s,\pi(s))
    \]
    is a group homomorphism (the group law on $\Pi_X = F \times \Gal(\bar{k}/k)$ being given by \eqref{eq:GroupLawOnPiX}) fitting in a commutative diagram of extensions
    \begin{equation} \label{eq:MorphismOfExtensions}
    \xymatrix{
        1 \ar[r] & \tensor[_a]{F}{} \ar[r] \ar@{=}[d] & \tensor[_a]{F}{} \rtimes \Gal(\bar{K}/K) \ar[r] \ar[d]^{\omega} & \Gal(\bar{K}/K) \ar[r] \ar[d]^\pi & 1 \\
        1 \ar[r] & F \ar[r] & \Pi_X \ar[r] & \Gal(\bar{k}/k) \ar[r] & 1.
    }
    \end{equation}
    The $K$-point $P \in X(K)$ gives rise to the section
    \[
        \sigma_s = (\fbf_s,s)
    \]
    of the top row of \eqref{eq:MorphismOfExtensions} (see e.g. \cite[\S{2.2}]{Lagarde2025Brauer}). Now, according to \cite[Proposition~3.4]{LA2019Brauer}, the evaluation $[\Acal_L](P) \in \H^2(K,\Fbf_p)$ is represented by the pullback 
    \[
        \Bcal = \sigma^\ast \omega^\ast  \Acal_L = \sigma^\ast \omega^\ast q^\ast \Psi_L \colon \Gal(\bar{K}/K) \times \Gal(\bar{K}/K) \longrightarrow \Fbf_p,
    \] 
    where $q$ is the homomorphism from Lemma~\ref{lemm:TheHomomorphismQ}. For $s,t \in \Gal(\bar{K}/K)$, we compute
        \begin{align*}
            \Bcal(s,t) & = \Psi_L(q(\omega(\sigma_s)), q(\omega(\sigma_t))) = \Psi_L(q(\fbf_s \tilde{a}_s,\pi(s)), q(\fbf_t \tilde{a}_t,\pi(t))) = \Psi_L(a'_s + a_s, a'_t + a_t) \\
            & = L((x'_s + x_s) \otimes (y'_t + y_t)),
        \end{align*}
    so $[\Acal_L](P) = L_\ast([x + x'] \cup [y + y'])$ as claimed.
\end{proof}


\subsection{Explicit computation of the Brauer--Manin obstruction} \label{subsec:TransBrauerExplicit}

So far, we have constructed a class $\Acal_L \in \Br(X)$ from any given linear form $L\colon M \otimes M^\vee \to \Fbf_p$, but we have not established its unramified nature. Rather than proving unramifiedness algebraically, we use Harari's evaluation-theoretic criterion.

We recall that the class $\beta \in \H^2(k,Z)$ and the places $v_1,\ldots,v_d \in \Omega$ were constructed so that
\[
    \Inv_{v_i}^Z(\beta) = \begin{cases}
        \phi(e_i \otimes e_i) & \text{if } v = v_i \quad (i = 1,\ldots,d),\\
        0 & \text{otherwise}.
    \end{cases} 
\]
Recall that the isomorphism $\Inv_v^{Z}$ was defined from \eqref{eq:LocalInvariantA}).

\begin{prop} \label{prop:Unramified}
    Let $[\Acal_L] \in \Br(X)$ be the class constructed from any linear form 
    \[
        L \colon M \otimes M^\vee = \Mat_d(\Fbf_p) \longrightarrow \Fbf_p
    \]
    as in Subsection~\ref{subsec:Evaluation}. Then the following are true.
    \begin{enumerate}
        \item \label{prop:Unramified:1} For any place $v \in \Omega$ and any $P_v \in X(k_v)$, one has
        \[
            \inv_v([\Acal_L](P_v)) = \begin{cases}
                L(e_i \otimes e_i) & \text{if } v = v_i, \quad (i = 1,\ldots,d), \\
                0 & \text{otherwise}
            \end{cases}
       \]
    (here, $\inv_v$ takes value in $\Zbf/p$).
        
        \item \label{prop:Unramified:2} $[\Acal_L]$ is unramified.
    \end{enumerate}
\end{prop}
\begin{proof}
    We prove~\ref{prop:Unramified:1}. By Proposition~\ref{prop:EvaluationOfAL}, there exists $x'_v \in \H^1(k_v,M)$ and $x'_v \in \H^1(k_v,M^\vee)$ with
    \[
        [\Acal_L](P_v) = L_\ast(x'_v \cup y'_v) \qquad \text{and} \qquad \phi_\ast(x'_v \cup y_v') = \beta|_{k_v}.
    \]
   If $v$ is a finite or real place, let $A_v':=\Inv_v^{M \otimes M^\vee}(x'_v \cup y'_v) \in M \otimes M^\vee$. We can now repeat the rank-bounding argument from the proof of Proposition~\ref{prop:NoRationalPoints}. Namely, $A'_v$ has small rank (bounded uniformly in $v$). On the other hand, it follows from the functoriality \eqref{eq:LocalInvariantFunctorial} that
    \begin{align*}
        \phi(A'_v) & = \phi(\Inv_v^{M \otimes M^\vee}(x'_v \cup y'_v)) = \Inv_v^Z(\phi_\ast(x'_v \cup y'_v)) \\
        & = \Inv_v^Z(\beta) = 
        \begin{cases}
            \phi(e_i \otimes e_i) & \text{if } v = v_i \quad (i=1,\ldots,d), \\
            0 & \text{otherwise}. 
        \end{cases}
    \end{align*}
    For $d$ sufficiently large (any $d \ge [k:\Qbf] + 3$ will do the trick), the rank bound on $A'_v$ again forces
    \[
        A'_v = \begin{cases}
            e_i \otimes e_i & \text{if } v = v_i \qquad (i=1,\ldots,d), \\
            0 & \text{otherwise}. 
        \end{cases}
    \]
    Claim~\ref{prop:Unramified:1} then follows from the functoriality condition \eqref{eq:LocalInvariantFunctorial}. In particular, the evaluation map induced by $[\Acal_L]$ is zero for all but finitely many places $v \in \Omega$. By \cite[Th\'eor\`eme~2.1.1]{Harari1994Fibration}, we have $[\Acal_L] \in \Br_{\nr}(X)$, which proves~\ref{prop:Unramified:2}. 
\end{proof}

With Proposition~\ref{prop:Unramified} at our disposal, we are now ready to produce unramified Brauer classes that obstruct the Brauer--Manin condition. Let $L\colon M \otimes M^\vee = \Mat_d(\Fbf_p) \to \Fbf_p$ be the $(1,1)$-entry and $[\Acal_L] \in \Br_{\nr}(X)$ the corresponding element. Then, for any $(P_v) \in \prod_{v \in \Omega} X(k_v)$, one has
\[
    \sum_{v \in \Omega} \inv_v([\Acal_L](P_v)) = \sum_{i=1}^d L(e_i \otimes e_i) = 1 \in \Zbf/p.
\]
In particular,  the class $[\Acal_L]$ is necessarily transcendental, because $\Br_{\nr,1}(X) = \Br_0(X)$ contributes trivially to the Brauer--Manin pairing.

\begin{remk}
    Here is an alternative way to see the transcendence of $[\Acal_L]$. Recall from Subsection~\ref{subsec:UnramifiedClass} that the restriction of $[\Acal_L]$ to $\Br(X_{\bar{k}}) = \H^2(F,\Qbf/\Zbf)$ is represented by the cocycle
    \[
        \bar{\Acal}_L \colon F \times F \longrightarrow \Fbf_p,
    \]            
    which itself is the inflation of the bilinear map
    \[
        \Psi_L\colon (M \oplus M^\vee) \times (M \oplus M^\vee) \longrightarrow \Fbf_p, \quad ((x,y),(x',y')) \longmapsto L(x \otimes y').
    \]
    Under the isomorphism 
    \[
        \H^2(M \oplus M^\vee, \Qbf/\Zbf) \cong \Hom\left(\bigwedge^2 (M \oplus M^\vee), \Qbf/\Zbf\right) = \Hom\left(\bigwedge^2 M \oplus (M \otimes M^\vee) \oplus \bigwedge^2 M^\vee, \Qbf/\Zbf\right),
    \]
    the class of $\Psi_L$ corresponds to that of
    \[
        \tilde{\Psi}_L \colon \bigwedge^2 M \oplus (M \otimes M^\vee) \oplus \bigwedge^2 M^\vee \to \Fbf_p, \qquad (\gamma_1, \gamma_2, \gamma_3) \longmapsto L(\gamma_2).
    \]
    Recall from \eqref{eq:CommutatorMapLambda} that the commutator map $\lambda \colon \bigwedge^2 M \oplus (M \otimes M) \oplus \bigwedge^2 M \to Z$ has kernel 
    \[
        \Ker(\lambda) = \bigwedge^2 M \oplus \Ker (\phi) \oplus \bigwedge^2 M^\vee,
    \]
    and the Bogomolov multiplier of $F$ is 
    \[
        \B_0(F) = \Hom(\Ker(\lambda)/S_\lambda,\Qbf/\Zbf) =  \Hom(\Ker(\phi),\Qbf/\Zbf),
    \]
    where $S_\lambda=\bigwedge^2 M \oplus \{0\} \oplus \bigwedge^2 M^\vee$, cf. \cite[Proof of Lemme~4.8 and Proposition~4.9]{Linh2024BK}. We have $\Ibf_d \in \Ker(\phi)$, but our choice of the functional $L$ gives $L(\Ibf_d) = 1 \neq 0$, so $\tilde{\Psi}_L$ remains nonzero when restricted to the subgroup $\Ker(\phi)$ of the second factor $M \otimes M^\vee$. Thus, $[\bar{\Acal}_L]$ is a nonzero element of $\Br_{\nr}(X_{\bar{k}})$, that is, $[\Acal_L]$ is a transcendental Brauer class.
\end{remk}